\documentclass{article}

\usepackage{amsmath}
\usepackage{amssymb}
\usepackage{amsfonts}
\usepackage{amsthm}
\usepackage{cleveref}
\usepackage{fullpage}
\usepackage{xcolor}
\usepackage{graphicx}

\newcommand{\R}{\mathbb{R}}
\newcommand{\N}{\mathbb{N}}

\renewcommand{\d}{\mathrm{d}}
\renewcommand{\div}{\mathrm{div}}

\newcommand{\mass}[1]{\mathbb M(#1)}
\newcommand{\flatnorm}[1]{\mathbb F(#1)}
\newcommand{\hd}{\mathcal{H}}
\newcommand{\restr}{{\mbox{\LARGE$\llcorner$}}}
\newcommand{\pushforward}[2]{{{#1}_{\#}#2}}
\newcommand{\meas}{{\mathcal M}}
\newcommand{\measp}{{\meas_+}}

\newcommand{\weakstarto}{\mathop{\stackrel\ast\rightharpoonup}}
\newcommand{\massfluxto}{\mathop{\stackrel{\mathrm f}\rightharpoonup}}
\newcommand{\flatto}{\mathop{\stackrel\flat\rightharpoonup}}
\newcommand{\flatChains}[1]{{\mathbf{F}_{#1}}}
\newcommand{\finiteFlatChains}[1]{{\overline{\mathbf{F}}_{#1}}}
\newcommand{\massFluxes}{{\mathcal F}}

\newcommand{\diff}{{\mathrm{diff}}}
\newcommand{\rec}{{\mathrm{rec}}}
\newcommand{\hMass}[2][h]{\mathbb M_{#1}(#2)}
\newcommand{\brTptCost}[2][h]{\mathbb J_{#1}(#2)}

\newcommand{\flux}{{\rho}}
\newcommand{\chain}{{T}}

\newtheorem{theorem}{Theorem}
\newtheorem{lemma}[theorem]{Lemma}
\newtheorem{corollary}[theorem]{Corollary}
\newtheorem{definition}[theorem]{Definition}

\newtheorem{remark}[theorem]{Remark}

\graphicspath{{./figures/},{.}}

\title{Approximation of rectifiable $1$-currents and weak-$\ast$ relaxation of the $h$-mass}
\author{Andrea Marchese \and Benedikt Wirth}
\date{}

\begin{document}

\maketitle

\begin{abstract}
Based on Smirnov's decomposition theorem
we prove that every rectifiable $1$-current $\chain$ with finite mass $\mass\chain$ and finite mass $\mass{\partial\chain}$ of its boundary $\partial\chain$ can be approximated in mass by a sequence of rectifiable $1$-currents $\chain_n$
with polyhedral boundary $\partial\chain_n$ and $\mass{\partial\chain_n}$ no larger than $\mass{\partial\chain}$. 
Using this result we can compute the relaxation of the $h$-mass for polyhedral $1$-currents with respect to the joint weak-$\ast$ convergence of currents and their boundaries.
We obtain that this relaxation coincides with the usual $h$-mass for normal currents.
This shows that the concepts of so-called generalized branched transport and the $h$-mass are equivalent.
\end{abstract}

\section{Introduction}
Variational models for ramified transportation networks have recently attracted lots of interest (see for instance \cite{Xi03,BeCaMo09,MaSo13,Sa15,BrWi18,CoDRMa18} and the references therein).
They are closely related to the measure-geometric concept of the $h$-mass of normal currents (as for instance introduced in \cite{Fl66}, where $h$ plays the role of a group metric).
The main difference is that the transportation network models are defined via relaxation with respect to weak-$\ast$ convergence of currents and their boundaries,
while the $h$-mass is defined via relaxation with respect to the weaker notion of flat convergence.
In \cite[Prop.\,2.32]{BrWi18} the equivalence between both models was used without proof.
In this note we prove the equivalence between generalized branched transport and the $h$-mass in full generality in \cref{thm:equivalence}.
The main tool will be a recent relaxation result by Chambolle, Ferrari, and Merlet \cite{ChFeMe18} for currents with polyhedral boundary,
combined with a new strong approximation result of rectifiable $1$-currents by currents with polyhedral boundary and equibounded boundary mass (\crefrange{thm:rectifiableBoundary}{thm:polyhedralBoundaryApprox}).
In the remainder of the introduction we describe the above-mentioned models and corresponding notions in more detail.

Following \cite{Xi03} or its generalization \cite{BrWi18}, the generalized branched transport model can be introduced as follows
(where our notation is chosen slightly differently to emphasize the correspondence to the $h$-mass later).
Throughout the article we consider $\Omega\subset\R^d$ to be the closure of an open bounded connected domain,
and we denote by $\meas(\Omega)$ the set of Radon measures, by $\measp(\Omega)\subset\meas(\Omega)$ the subset of nonnegative measures, and by $\meas(\Omega;\R^d)$ the set of $\R^d$-valued Radon measures on $\Omega$.
The total variation measure will be indicated by $|\cdot|$, the total variation of a measure by $\|\cdot\|_\meas$, and weak-$\ast$ convergence by $\weakstarto$.
The notation $\hd^m$ denotes the $m$-dimensional Hausdorff measure and $\restr$ the restriction of measures to Borel sets.
\begin{definition}[Generalized branched transport]
\begin{enumerate}
\item
A \emph{unit line flux} along $e$ is a measure $\flux\in\meas(\Omega;\R^d)$ of the form $\flux=\vec e\hd^1\restr e$,
where $e\subset\Omega$ is a straight line segment with unit tangent $\vec e$.
\item
A \emph{polyhedral flux} in $\Omega$ is a measure $\flux\in\meas(\Omega;\R^d)$ of the form $\flux=\sum_{i=1}^na_i\flux_i$ for $n\in\N$, $a_1,\ldots,a_n\in\R$, and $\flux_1,\ldots,\flux_n$ unit line fluxes.
\item
If the distributional divergence of $\flux\in\meas(\Omega;\R^d)$ is a Radon measure, then $\flux$ is called a \emph{mass flux}, and its negative divergence is called the \emph{boundary of $\flux$}, denoted by $\partial\flux=-\div\flux$.
The set of mass fluxes on $\Omega$ is denoted $\massFluxes(\Omega)$.
\item
Let $\flux_1,\flux_2,\ldots$ be a sequence of mass fluxes.
We say $\flux_n$ \emph{converges weakly} to mass flux $\flux$ and write $\flux_n\massfluxto\flux$ as $n\to\infty$,
if $\flux_n\weakstarto\flux$ and $\partial\flux_n\weakstarto\partial\flux$.
\item
A \emph{transportation cost} is a subadditive, nondecreasing, lower semi-continuous function $h:[0,\infty)\to[0,\infty)$ with $h(0)=0$.
\item
Given a transportation cost $h$, the corresponding \emph{generalized branched transport cost} of a polyhedral flux $\flux=\sum_{i=1}^na_i\vec e_i\hd^1\restr e_i$ with non-overlapping line segments $e_i$ is
\begin{equation*}
\brTptCost\flux=\sum_{i=1}^nh(|a_i|)\hd^1(e_i)\,.
\end{equation*}
The \emph{generalized branched transport cost} of a mass flux $\flux$ is
\begin{equation*}
\brTptCost\flux=\inf\left\{\liminf_{n\to\infty}\brTptCost{\flux_n}\,\middle|\,\flux_n\in\massFluxes(\Omega)\text{ polyhedral, }\flux_n\massfluxto\flux\text{ as }n\to\infty\right\}\,,
\end{equation*}
the relaxation of the generalized branched transport cost on polyhedral fluxes with respect to weak convergence of mass fluxes.
\end{enumerate}
\end{definition}
The variational problem of finding optimal mass transportation schemes between a given mass source $\mu_+\in\measp(\Omega)$ and a sink $\mu_-\in\measp(\Omega)$ then is
\begin{equation*}
\min\left\{\brTptCost\flux\,\middle|\,\flux\text{ is mass flux with }\partial\flux=\mu_--\mu_+\right\}\,.
\end{equation*}
The existence of minimizers and their properties are discussed in \cite{BrWi18}.
Note that mass fluxes are also known as \emph{divergence measure vector fields} \cite{Si07} or \emph{vector charges} \cite{Sm93} or \emph{$1$-dimensional normal currents} \cite{Fe69}.

The definition of the $h$-mass of a flat chain follows the same strategy.
\begin{definition}[$h$-mass of a flat chain]
\begin{enumerate}
\item
An \emph{$m$-dimensional polyhedron} in $\Omega$ is an oriented polyhedral subset of an $m$-dimensional plane $H\subset\Omega$ with nonempty relative interior.
\item
A \emph{polyhedral $m$-chain} in $\Omega$ is a linear combination $\chain=\sum_{i=1}^na_ie_i$
with $n\in\N$, $a_1,\ldots,a_n\in\R$, and $e_1,\ldots,e_n$ $m$-dimensional polyhedra in $\Omega$.
A refinement of $\chain$ is a polyhedral $m$-chain of the form $\sum_{i=1}^n\sum_{k=1}^{K_i}a_ie_i^k$, where $e_i=e_i^1\cup\ldots\cup e_i^{K_i}$ represents a disjoint partition of $e_i$.
Two polyhedral $m$-chains are \emph{equivalent} and identified with each other, if they have a joint refinement.
When writing a polyhedral $m$-chain as $\chain=\sum_{i=1}^na_ie_i$ we shall always tacitly assume the $e_i$ to be pairwise disjoint (which can always be achieved).
\item
The \emph{boundary} of a polyhedral $m$-chain $\chain=\sum_{i=1}^na_ie_i$ is the polyhedral $(m-1)$-chain $\partial\chain=\sum_{i=1}^na_i\partial e_i$,
where $\partial e_i$ is the sum of the oriented faces in the relative boundary of $e_i$.
\item
The \emph{mass} of a polyhedral $m$-chain $\chain=\sum_{i=1}^na_ie_i$ is $\mass\chain=\sum_{i=1}^n|a_i|\hd^m(e_i)$.
\item
The \emph{flat norm} of a polyhedral $m$-chain reads
\begin{equation*}
\flatnorm\chain=\inf\{\mass{\chain-\partial D}+\mass D\,|\,D\text{ is polyhedral }(m+1)\text{-chain}\}\,.
\end{equation*}
\item
The Banach space $\flatChains{m}(\Omega)$ of \emph{flat $m$-chains} is the completion of the vector space of polyhedral $m$-chains in $\Omega$ under the flat norm.
The linear boundary operator $\partial$ is extended continuously with respect to the flat norm onto all of $\flatChains{m}(\Omega)$.
The mass functional $\mass\cdot$ is extended onto $\flatChains{m}(\Omega)$ via relaxation with respect to the flat norm.
The subspace of flat $m$-chains of finite mass and with finite mass boundary is denoted $\finiteFlatChains{m}(\Omega)\subset\flatChains{m}(\Omega)$.
\item
Let $\chain_1,\chain_2,\ldots$ be a sequence of flat $m$-chains.
We say $\chain_n$ \emph{converges in mass} to the flat $m$-chain $\chain$ and write $\chain_n\to\chain$, if $\mass{\chain_n-\chain}\to0$.
We say $\chain_n$ \emph{converges flatly} to $\chain$ and write $\chain_n\flatto\chain$, if $\flatnorm{\chain_n-\chain}\to0$.
\item
Given a transportation cost $h$, the corresponding \emph{$h$-mass} of a polyhedral $1$-chain $\chain=\sum_{i=1}^na_ie_i$ is
\begin{equation*}
\hMass\chain=\sum_{i=1}^nh(|a_i|)\hd^1(e_i)\,.
\end{equation*}
The \emph{$h$-mass} of a flat $1$-chain $\chain$ is
\begin{equation*}
\hMass\chain=\inf\left\{\liminf_{n\to\infty}\hMass{\chain_n}\,\middle|\,\chain_n\in\flatChains1(\Omega)\text{ polyhedral, }\chain_n\flatto\chain\text{ as }n\to\infty\right\}\,,
\end{equation*}
the relaxation of the $h$-mass on polyhedral fluxes with respect to flat convergence.
\end{enumerate}
\end{definition}
Note that flat $m$-chains with finite mass and boundary mass are also known as normal $m$-currents \cite[4.1.23 \& 4.2.23]{Fe69},\cite[Rem.\,2.29(2)]{BrWi18}.

The following remark details how mass fluxes and flat $1$-chains relate to each other.
In particular, flat convergence of flat $1$-chains is a strictly weaker notion than weak convergence of mass fluxes,
which is why in general the $h$-mass must be less than or equal to the generalized branched transport cost.

\begin{remark}[Flat $1$-chains and mass fluxes]
\begin{enumerate}
\item
Polyhedral fluxes $\flux=\sum_{i=1}^na_i\vec e_i\hd^1\restr e_i$ and polyhedral $1$-chains $\chain=\sum_{i=1}^na_ie_i$ can naturally be identified with each other.
Analogously, there is an obvious natural identification between polyhedral $0$-chains and finite discrete measures on $\Omega$.
\item
The identification between polyhedral $0$-chains and discrete measures can be extended to an isomorphism $\iota_0:\finiteFlatChains0(\Omega)\to\meas(\Omega)$.
Likewise, the identification between polyhedral 1-chains and polyhedral fluxes can be extended to an isomorphism $\iota_1:\finiteFlatChains1(\Omega)\to\massFluxes(\Omega)$.

The isomorphisms are consistent with the notions of boundary and convergence in the following sense.
\begin{enumerate}
\item
For any $\chain\in\finiteFlatChains1(\Omega)$ and $\flux\in\massFluxes(\Omega)$ with $\iota_1(\chain)=\flux$ we have $\iota_0(\partial\chain)=\partial\flux$.
\item
Let $\mu,\mu_1,\mu_2,\ldots\in\meas(\Omega)$ and $\flux,\flux_1,\flux_2,\ldots\in\massFluxes(\Omega)$, then
\begin{align*}
&\mu_n\mathop\to_{n\to\infty}\mu\text{ strongly in }\meas(\Omega)&\hspace*{-1.5ex}\text{is equivalent to}\qquad&\iota_0^{-1}(\mu_n)\mathop\to_{n\to\infty}\iota_0^{-1}(\mu)\text{ in mass, and}\\
&\flux_n\mathop\to_{n\to\infty}\flux\text{ strongly in }\massFluxes(\Omega)&\hspace*{-1.5ex}\text{is equivalent to}\qquad&\iota_1^{-1}(\flux_n)\mathop\to_{n\to\infty}\iota_1^{-1}(\flux)\text{ in mass.}
\end{align*}
\item
Let $\flux\in\massFluxes(\Omega)$ and $\flux_1,\flux_2,\ldots$ be a sequence in $\massFluxes(\Omega)$, then
\begin{equation*}
\flux_n\massfluxto_{n\to\infty}\flux
\quad\text{implies}\quad
\iota_1^{-1}(\flux_n)\flatto_{n\to\infty}\iota_1^{-1}(\flux)\,.
\end{equation*}
Vice versa, let $\chain_1,\chain_2,\ldots\in\finiteFlatChains1(\Omega)$ have equibounded mass and boundary mass and let $\chain\in\flatChains1(\Omega)$, then
\begin{equation*}
\chain_n\flatto_{n\to\infty}\chain
\quad\text{implies}\quad
\iota_1(\chain_n)\massfluxto_{n\to\infty}\iota_1(\chain)\,.
\end{equation*}
Analogously, $\mu_n\weakstarto\mu$ in $\meas(\Omega)$ as $n\to\infty$ implies $\iota_0^{-1}(\mu_n)\flatto\iota_0^{-1}(\mu)$,
while $\chain_n\flatto\chain$ in $\finiteFlatChains0(\Omega)$ as $n\to\infty$ implies $\iota_0(\chain_n)\weakstarto\iota_0(\chain)$ under the condition that the flat $0$-chains $\chain_n$ have equibounded mass.
\end{enumerate}
The proof of the above essentially relies on weak-$\ast$ compactness of measures with bounded mass and classical deformation theorems such as \cite{Wh99b};
for more details see the brief summary in \cite[Rem.\,2.29]{BrWi18} and the references therein.
\item
Example sequences of flat (polyhedral) $1$-chains in $[-1,1]^2$ that converge flatly to $0$, while the corresponding mass fluxes do not converge weakly, are
\begin{align*}
\chain_n&=\sum_{k=-n}^{n-1}\left[\left(\tfrac kn,0\right),\left(\tfrac kn+\tfrac1{n^2},0\right)\right]
\quad\text{and}\\
S_n&=n[c_1,c_2]+n[c_2,c_3]+n[c_3,c_4]+n[c_4,c_1]\\
&\quad-n[\alpha c_1,\alpha c_2]-n[\alpha c_2,\alpha c_3]-n[\alpha c_3,\alpha c_4]-n[\alpha c_4,\alpha c_1]
\quad\text{with }\alpha=1-\tfrac1{n^2}\,,
\end{align*}
where $[a,b]$ denotes the line segment from $a$ to $b$ and $c_1,\ldots,c_4$ denote the four corners of $[-1,1]^2$.
While $\iota_1(\chain_n)\to0$ strongly in $\meas([-1,1]^2;\R^2)$, $\partial(\iota_1(\chain_n))$ diverges in $\meas([-1,1]^2)$.
On the other hand, $\partial(\iota_1(S_n))=0$ for all $n$, but $\iota_1(S_n)$ diverges in $\meas([-1,1]^2;\R^2)$.
\end{enumerate}
\end{remark}

\begin{corollary}[Bound of $h$-mass by branched transport cost]
Let $\chain\in\finiteFlatChains1(\Omega)$ and $h$ be a transportation cost, then $\hMass\chain\leq\brTptCost{\iota_1(\chain)}$.
\end{corollary}
In this note we show equality.
\begin{theorem}[Equivalence of $h$-mass and branched transport cost]\label{thm:equivalence}
Let $\chain\in\finiteFlatChains1(\Omega)$ and $h$ be a transportation cost, then $\hMass\chain=\brTptCost{\iota_1(\chain)}$.
\end{theorem}
The proof will be provided in \cref{sec:relaxation}.
It will be based on the following lemmas for $1$-rectifiable flat chains, whose statement requires the notion of rectifiability and acyclicity introduced below.
\begin{definition}[Rectifiable and acyclic mass fluxes and flat chains]
\begin{enumerate}
\item
Let $m\in\{0,1\}$. Given a Borel set $A\subset\Omega$ and a flat $m$-chain $\chain\in\finiteFlatChains{m}(\Omega)$ with $\iota_m(\chain)\restr A\in\massFluxes(\Omega)$,
the \emph{restriction of $\chain$ to $A$} is defined as $\chain\restr A=\iota_m^{-1}(\iota_m(\chain)\restr A)$.
The restriction to $A$ can be extended to all flat $m$-chains of finite mass by continuity with respect to flat convergence.
\item
A subset $\Sigma\subset\R^d$ is called \emph{$m$-rectifiable} if it is contained in the countable union of $m$-dimensional $C^1$-submanifolds, up to an $\hd^m$-negligible set.
\item
A (vector-valued) Radon measure $\flux$ or a flat $m$-chain $\chain$ are called \emph{$m$-rectifiable}
if there exists an $m$-rectifiable set $\Sigma\subset\R^d$ with $\flux=\flux\restr\Sigma$ or $\chain=\chain\restr\Sigma$, respectively.
\item
A mass flux $\flux\in\massFluxes(\Omega)$ is called \emph{acyclic}
if it cannot be decomposed into $\flux=\flux^a+\flux^b$ with $\flux^b\neq0$, $\partial\flux^b=0$, and $\|\flux\|_\meas=\|\flux^a\|_\meas+\|\flux^b\|_\meas$.
A flat $1$-chain $\chain\in\finiteFlatChains1(\Omega)$ is \emph{acyclic} if $\iota_1(\chain)$ is.
\end{enumerate}
\end{definition}
Note that the restriction for flat chains can also be defined without reference to mass fluxes as in \cite{Fl66}.
\begin{lemma}[Approximation of rectifiable mass fluxes by mass fluxes with rectifiable boundary]\label{thm:rectifiableBoundary}
Let $\flux\in\massFluxes(\Omega)$ be $1$-rectifiable and acyclic,
then there exists a monotonically increasing sequence of $\flux$-measurable functions $\lambda_1,\lambda_2,\ldots:\Omega\to[0,1]$
and associated $1$-rectifiable acyclic mass fluxes $\flux_1=\lambda_1\flux,\flux_2=\lambda_2\flux,\ldots$
with $\flux_n\to\flux$ strongly as $n\to\infty$,
where $\partial\flux_n$ is $0$-rectifiable with $\|\partial\flux_n\|_\meas\leq\|\partial\flux\|_\meas$ for all $n\in\N$.
\end{lemma}
\begin{lemma}[Approximation of mass fluxes by fluxes with finite discrete boundary]\label{thm:polyhedralBoundary}
Let $\flux\in\massFluxes(\Omega)$ be acyclic with $0$-rectifiable boundary $\partial\flux$,
then there exists a monotonically increasing sequence of $\flux$-measurable functions $\kappa_1,\kappa_2,\ldots:\Omega\to[0,1]$
and associated acyclic mass fluxes $\flux_1=\kappa_1\flux,\flux_2=\kappa_2\flux,\ldots$
with $\flux_n\to\flux$ strongly as $n\to\infty$,
where $\partial\flux_n$ has finite support and $\|\partial\flux_n\|_\meas\leq\|\partial\flux\|_\meas$ for all $n\in\N$.
\end{lemma}
The proof of both lemmas will be provided in \cref{sec:lemmas}.
A direct consequence is the following \namecref{thm:polyhedralBoundaryApprox}.
\begin{lemma}[Approximation of rectifiable $1$-chains by $1$-chains with polyhedral boundary]\label{thm:polyhedralBoundaryApprox}
Let $h$ be a transportation cost.
For any rectifiable $\chain\in\finiteFlatChains1(\Omega)$ there exists a sequence $\chain_1,\chain_2,\ldots\in\finiteFlatChains1(\Omega)$
with $\chain_n\to\chain$ in mass as $n\to\infty$
such that $\hMass{\chain_n}\to\hMass\chain$ as $n\to\infty$
and such that $\partial\chain_n$ is polyhedral with $\mass{\partial\chain_n}\leq\mass{\partial\chain}$ as well as $\hMass[\tilde h]{\chain_n}\leq\hMass[\tilde h]\chain$ for all $n\in\N$ and transportation costs $\tilde h$.
\end{lemma}
\begin{proof}
First note that by White's structure theorem \cite[Sec.\,6]{Wh99b}
one can identify any $1$-rectifiable flat chain $\chain\in\finiteFlatChains1(\Omega)$
with a triple $[\Sigma,\theta,m]$ of a $1$-rectifiable set $\Sigma\subset\Omega$ with approximate tangent $\theta:\Sigma\to\mathbb S^{d-1}$ and a measurable function $m:\Sigma\to\R$ such that
\begin{equation*}
\iota_1(\chain)=m\theta\hd^1\restr\Sigma\,.
\end{equation*}
Again by \cite[Sec.\,6]{Wh99b}, its $h$-mass in this case can be expressed as
\begin{equation*}
\hMass\chain=\int_\Sigma h(|m|)\,\d\hd^1\,.
\end{equation*}

Now consider $\flux=\iota_1(\chain)=m\theta\hd^1\restr\Sigma$.
By Smirnov's decomposition theorem \cite[Thm.\,C]{Sm93} we can decompose $\flux=\flux^a+\flux^b$, where $\flux^a=m^a\theta\hd^1\restr\Sigma\in\massFluxes(\Omega)$ is rectifiable and acyclic with $\partial\flux^a=\partial\flux$ and $\flux^b=m^b\theta\hd^1\restr\Sigma\in\massFluxes(\Omega)$ (with $m^b$ having the same sign as $m^a$ pointwise) satisfies $\partial\flux^b=0$.
By \cref{thm:rectifiableBoundary} there is a monotonically increasing sequence $\lambda_1,\lambda_2,\ldots:\Omega\to[0,1]$ of $\flux^a$-measurable functions
such that $\lambda_n\to1$ monotonically $\flux^a$-almost everywhere
and such that $\partial(\lambda_n\flux^a)$ is $0$-rectifiable with $\|\partial(\lambda_n\flux^a)\|_\meas\leq\|\partial\flux^a\|_\meas=\|\partial\flux\|_\meas$ for all $n\in\N$.
Denote the flat $1$-chains corresponding to $\flux^a_n+\flux^b$ by $\tilde\chain_n=[\Sigma,\theta,\lambda_nm^a+m^b]$.
By the Monotone Convergence Theorem
\begin{equation*}
\hMass{\tilde\chain_n}
=\int_\Sigma h(|\lambda_nm^a+m^b|)\,\d\hd^1
\to\int_\Sigma h(|m^a+m^b|)\,\d\hd^1
=\hMass{\chain}
\end{equation*}
as $n\to\infty$ so that (potentially after passing to a subsequence) we may assume
\begin{equation*}
\mass{\tilde\chain_n-\chain}=\|\flux^a_n-\flux^a\|_\meas\leq\frac1n
\qquad\text{and}\qquad
\left|\hMass{\tilde\chain_n}-\hMass{\chain}\right|\leq\frac1n\,.
\end{equation*}
Similarly, appealing to \cref{thm:polyhedralBoundary} instead of \cref{thm:rectifiableBoundary},
for each $n\in\N$ there is some sequence $\chain_{n,k}=[\Sigma,\theta,\kappa_k\lambda_nm^a+m^b]$, $k=1,2,\ldots$, with $\partial\chain_{n,k}$ polyhedral, $\mass{\partial\chain_{n,k}}\leq\mass{\partial\tilde\chain_n}\leq\mass\chain$, and
\begin{equation*}
\mass{\chain_{n,k}-\tilde\chain_n}=\|\kappa_k\flux^a_n-\flux^a_n\|_\meas\leq\frac1k
\qquad\text{and}\qquad
\left|\hMass{\chain_{k,n}}-\hMass{\tilde\chain_n}\right|\leq\frac1k\,.
\end{equation*}
Thus, the sequence $\chain_n=\chain_{n,n}$ has all desired properties.
\end{proof}


\section{Proof of main lemmas}\label{sec:lemmas}
The proof uses Smirnov's decomposition theorem, part of which we restate for convenience.

\begin{definition}[Simple oriented curve]
A \emph{simple oriented curve of finite length} in $\Omega$ is a mass flux of the form
\begin{equation*}
\tilde\flux=\pushforward\gamma{\dot\gamma\hd^1\restr[0,1]}\,,
\end{equation*}
where $\pushforward\gamma\mu$ denotes the pushforward of a measure $\mu$ under a function $f$ and $\gamma:[0,1]\to\Omega$ is an injective Lipschitz curve.
Note that, writing $\delta_x$ for the Dirac mass at $x$,
\begin{equation*}
\partial\tilde\flux=\delta_{\gamma(1)}-\delta_{\gamma(0)}\,.
\end{equation*}
\end{definition}

\begin{theorem}[\protect{Smirnov's decomposition theorem, \cite[Thm.\,B-C]{Sm93}}]
For any acyclic $\flux\in\massFluxes(\Omega)$ there is a set $J$ of simple oriented curves of finite length and a nonnegative measure $\mu$ on $J$ such that
\begin{align*}
\flux&=\int_J\tilde\flux\,\d\mu(\tilde\flux)\,,\\
\|\flux\|_\meas&=\int_J\|\tilde\flux\|_\meas\,\d\mu(\tilde\flux)\,,\\
\|\partial\flux\|_\meas&=\int_J\|\partial\tilde\flux\|_\meas\,\d\mu(\tilde\flux)\,.
\end{align*}
Above, the first line means
$$\langle\flux,v\rangle=\int_J\langle\tilde\flux,v\rangle\,\d\mu(\tilde\flux)$$
for every smooth test vector field $v:\Omega\to\R^d$, where $\langle\cdot,\cdot\rangle$ denotes the dual pairing between vector-valued Radon measures and continuous vector fields on $\Omega$.
\end{theorem}

We will furthermore use the following two simple results about the generic intersection between a regular grid and a rectifiable set and a rectifiable mass flux.

\begin{lemma}[Rectifiable set and grid]\label{thm:grid}
Let $\Sigma\subset\Omega$ be $1$-rectifiable and define the rectilinear grid
\begin{equation*}
\mathcal G_n=\left\{y\in\R^d\,\middle|\,y_i=\tfrac mn\text{ for some }m\in\N\text{ and }i\in\{1,\ldots,d\}\right\}
\end{equation*}
of grid width $\frac1n$.
Then for almost every $x\in\R^d$ the set $\Sigma\cap(x+\mathcal G_n)$ is countable for all $n\in\N$.
\end{lemma}
\begin{proof}
Since $\mathcal G_n=\bigcup_{i=1}^d\mathcal S^i_n$ for
\begin{equation*}
\mathcal S^i_n=\left\{y\in\R^d\,\middle|\,y_i=\tfrac mn\text{ for some }m\in\N\right\}\,,
\end{equation*}
it suffices to show for fixed $1\leq i\leq d$ that for almost every $s\in\R$
the set $\Sigma\cap(se_i+\mathcal S^i_n)$ is countable for all $n\in\N$ (here $e_i$ denotes the $i$\textsuperscript{th} Cartesian unit vector).
To this end it suffices to show that for almost all $r\in\R$ the intersection of $\Sigma$ with the hyperplane
\begin{equation*}
\mathcal P_r^i=\left\{y\in\R^d\,\middle|\,y_i=r\right\}
\end{equation*}
is countable.
Indeed, $\mathcal S=\bigcup_{n=1}^\infty\mathcal S^i_n$ can be expressed as a countable union $\bigcup_{j=1}^\infty\mathcal P_{r_j}^i$ of such hyperplanes;
thus the set of $s\in\R$ for which $\Sigma\cap(se_i+\mathcal S)$ is uncountable is given by $\bigcup_{j=1}^\infty R_j$ with
\begin{equation*}
R_j=\left\{s\in\R\,\middle|\,\Sigma\cap\mathcal P_{s+r_j}^i\text{ is uncountable}\right\}\,,
\end{equation*}
which we show to be a nullset below.
Consequently, $\Sigma\cap(se_i+\mathcal S)$ is countable for almost all $s\in\R$.

To show that $\Sigma\cap\mathcal P_r^i$ is countable for almost all $r\in\R$
it suffices to cite the coarea formula for rectifiable sets \cite[3.2.22(2) with $W=\Sigma$, $f(x)=x_i$]{Fe69}
which states that $\Sigma\cap\mathcal P_r^i$ is $\hd^0$-measurable and $\hd^0$-rectifiable for almost all $r\in\R$.
\end{proof}

\begin{lemma}[Smirnov curves and grid]\label{thm:gridCurve}
Let $\flux\in\massFluxes(\Omega)$ be acyclic and $1$-rectifiable so that $\flux=\flux\restr\Sigma$ for a $1$-rectifiable set $\Sigma\subset\Omega$ and there exists a decomposition
\begin{equation*}
\flux=\int_J\tilde\flux\,\d\mu(\tilde\flux)
\end{equation*}
into simple oriented curves by Smirnov's decomposition theorem.
Then for almost all $x\in\R^d$, $\mu(J_x^n)=0$ for all $n\in\N$ with 
\begin{equation*}
J_x^n=\left\{\tilde\flux=\pushforward\gamma{\dot\gamma\hd^1\restr[0,1]}\in J\,\middle|\,\gamma([0,1])\cap(x+\mathcal G_n)\not\subset\Sigma\right\}\,,
\end{equation*}
that is, for any $n\in\N$ the intersection of $\mu$-almost every Smirnov curve with $(x+\mathcal G_n)$ lies in $\Sigma$.
\end{lemma}
\begin{proof}
Obviously it suffices to prove the statement for fixed $n\in\N$, which we shall assume in the following.
Below, we will denote the Lipschitz curve associated with a simple oriented curve $\tilde\flux\in J$ by $\gamma_{\tilde\flux}$
and the complement of $\Sigma$ by $\Sigma^c$.

\emph{Step\,1.}
We first show for any Lipschitz curve $\gamma:[0,1]\to\Omega$ that
\begin{equation*}
\hd^d(A^\gamma)>0
\text{ implies }
\hd^1(\gamma([0,1])\setminus\Sigma)>0\,,
\end{equation*}
where
\begin{equation*}
A^\gamma=\{x\in\R^d\,|\,\gamma([0,1])\cap(x+\mathcal G_n)\not\subset\Sigma\}\,.
\end{equation*}
Indeed, assume $\hd^d(A^\gamma)>0$ and let $x\in A^\gamma$ be such that $\gamma(0)\notin(x+\mathcal G_n)$ and such that for $i=1,\dots,d$ the set
\begin{equation*}
A^\gamma_i(x)=\{s\in\R\,|\,(x_1,\dots,x_{i-1},s,x_{i+1},\dots,x_d)\in A^\gamma\}\,.
\end{equation*}
has Lebesgue density $1$ in the coordinate $x_i$.
Now consider the point
\begin{equation*}
y^x=\gamma(t^x)
\qquad\text{with }
t^x=\min\{t\in[0,1]\,|\,\gamma(t)\in(x+\mathcal G_n)\setminus\Sigma\}
\end{equation*}
and denote by $i\in\{1,\dots,d\}$ the index such that $y^x_i$ is the $i$\textsuperscript{th} coordinate of one of the hyperplanes of the grid $x+\mathcal G_n$.
Assume without loss of generality that $\gamma(0)_i<y^x_i$.
It follows that for $\varepsilon>0$ sufficiently small and for every $s\in A_i^\gamma\cap [x_i-\varepsilon,x_i]$,
the point $y^{x(s)}$ with $x(s)=(x_1,\dots,x_{i-1},s,x_{i+1},\dots,x_d)$ satisfies $y^{x(s)}_i=s$. Now
\begin{multline*}
\hd^1(\gamma([0,1])\setminus\Sigma)
\geq\hd^1\left(\left\{y^{x(s)}\in\R^d\,\middle|\,s\in A_i^\gamma\cap [x_i-\varepsilon,x_i]\right\}\right)\\
\geq\hd^1\left(\left\{y^{x(s)}_i\in\R\,\middle|\,s\in A_i^\gamma\cap [x_i-\varepsilon,x_i]\right\}\right)
=\hd^1(A_i^\gamma\cap [x_i-\varepsilon,x_i])
>0\,.
\end{multline*}

\emph{Step\,2.}
We show $\hd^d(A^{\gamma_{\tilde\flux}})=0$ for $\mu$-almost every $\tilde\flux\in J$.
Indeed, we have
\begin{equation*}
0
=\|\flux\|_\meas-\|\flux\restr\Sigma\|_\meas
\geq\int_J\|\tilde\flux\|_\meas\,\d\mu(\tilde\flux)-\int_J\|\tilde\flux\restr\Sigma\|_\meas\,\d\mu(\tilde\flux)
=\int_J\|\tilde\flux\restr\Sigma^c\|_\meas\,\d\mu(\tilde\flux)\,,
\end{equation*}
thus $\mu$-almost every $\tilde\flux\in J$ satisfies
\begin{equation*}
\hd^1(\gamma_{\tilde\flux}([0,1])\setminus\Sigma)
=\int_0^1\boldsymbol1_{\Sigma^c}(\gamma_{\tilde\flux})|\dot\gamma_{\tilde\flux}|\,\d\hd^1
=\int_{\Sigma^c}\,\d|\tilde\flux|
=\|\tilde\flux\restr\Sigma^c\|_\meas
=0\,,
\end{equation*}
where $\boldsymbol1_{\Sigma^c}$ is the characteristic function of $\Sigma^c$.
By the previous step this implies the desired result.

\emph{Step\,3.}
Finally we show $\hd^d(\{x\in\R^d\,|\,\mu(J_x^n)>0\})=0$, which concludes the proof.
Indeed, let us introduce the function
\begin{equation*}
h:\R^d\times J\to\{0,1\}\,,\quad
h(x,\tilde\flux)=
\begin{cases}
1 &\text{if }\tilde\flux\in J_x^n,\\
0 &\text{otherwise,}
\end{cases}
\end{equation*}
then by Fubini's theorem we have
\begin{equation*}
\int_{\R^d}\mu(J_x^n)\,\d x
=\int_J\int_{\R^d}h(x,\tilde\flux)\,\d x\,\d\mu(\tilde\flux)
=\int_{J}\hd^d(A^{\gamma_{\tilde\flux}})\,\d\mu(\tilde\flux)\,,
\end{equation*}
which is zero by the previous step.
Thus, $\mu(J_x^n)=0$ for almost all $x\in\R^d$, as desired.
\end{proof}

\begin{proof}[Proof of \cref{thm:rectifiableBoundary}]
Since $\flux$ is rectifiable, there is a $1$-rectifiable set $\Sigma\subset\Omega$ with $\flux=\flux\restr\Sigma$.
Using Smirnov's decomposition theorem we decompose $\flux$ into simple oriented curves,
\begin{equation*}
\flux=\int_J\tilde\flux\,\d\mu(\tilde\flux)\,.
\end{equation*}
Now, for $n\in\N$ consider the rectilinear grids $\mathcal G_{2^n}$ from \cref{thm:grid} with grid size $2^{-n}$
and note $\mathcal G_{2^n}\subset\mathcal G_{2^m}$ for $m\geq n$.
Since $\Sigma$ is $1$-rectifiable, by \cref{thm:grid,thm:gridCurve} there exists $x\in\Omega$ such that for all $n\in\N$ the intersection $(x+\mathcal G_{2^n})\cap\Sigma$ is countable
and $\mu$-almost all $\tilde\flux$ intersect $(x+\mathcal G_{2^n})$ in points which belong to $\Sigma$.
Now define for each simple oriented curve $\tilde\flux=\pushforward\gamma{\dot\gamma\hd^1\restr[0,1]}$ the pruned curve
\begin{equation*}
\tilde\flux_n=\pushforward\gamma{\dot\gamma\hd^1\restr[t_n^{\gamma,l},t_n^{\gamma,r}]}
\quad\text{ with }\quad
t_n^{\gamma,l}=\min\{t\in[0,1]\,|\,\gamma(t)\in x+\mathcal G_{2^n}\}\,,\quad
t_n^{\gamma,r}=\max\{t\in[0,1]\,|\,\gamma(t)\in x+\mathcal G_{2^n}\}
\end{equation*}
(if $\gamma$ does not intersect $x+\mathcal G_{2^n}$ we shall define $\tilde\flux_n=0$ by convention).
Next set
\begin{equation*}
\flux_n=\int_J\tilde\flux_n\,\d\mu(\tilde\flux)\,.
\end{equation*}
Using the properties of the Smirnov decomposition we obtain
\begin{equation*}
\|\flux\|_\meas
=\int_J\|\tilde\flux\|_\meas\,\d\mu(\tilde\flux)
=\int_J\|\tilde\flux_n\|_\meas+\|\tilde\flux-\tilde\flux_n\|_\meas\,\d\mu(\tilde\flux)
\geq\|\flux_n\|_\meas+\|\flux-\flux_n\|_\meas\,.
\end{equation*}
Together with the triangle inequality this implies $\|\flux\|_\meas=\|\flux_n\|_\meas+\|\flux-\flux_n\|_\meas$,
which in turn implies equality of the total variation measures, $|\flux|=|\flux_n|+|\flux-\flux_n|$, as well as parallelism of the Radon--Nikodym derivatives $\frac{\d\flux}{\d|\flux|}$ and $\frac{\d\flux_n}{\d|\flux_n|}$.
Consequently,
\begin{equation*}
\flux_n=\lambda_n\flux
\end{equation*}
for some $\flux$-measurable $\lambda_n:\Omega\to[0,1]$.
Replacing $\flux$ with $\flux_m$ for $m>n$ in the above argument implies $\flux_n=\lambda_{n,m}\flux_m$ for some $\flux$-measurable $\lambda_{n,m}:\Omega\to[0,1]$
so that $\lambda_n=\lambda_{n,m}\lambda_m\leq\lambda_m$.
Furthermore,
\begin{equation*}
\|\flux-\flux_n\|_\meas
=\left\|\int_J\tilde\flux-\tilde\flux_n\,\d\mu(\tilde\flux)\right\|_\meas
\leq\int_J\|\tilde\flux-\tilde\flux_n\|_\meas\,\d\mu(\tilde\flux)
\to0
\end{equation*}
by the Monotone Convergence Theorem.
Finally, using that $\partial\tilde\flux_n=(\partial\tilde\flux_n)\restr(\Sigma\cap(x+\mathcal G_{2^n}))$ for $\mu$ almost all $\tilde\flux$ due to our choice of $x$, we can compute
\begin{equation*}
\partial\flux_n\restr(\Sigma\cap(x+\mathcal G_{2^n}))
=\left(\int_J\partial\tilde\flux_n\,\d\mu(\tilde\flux)\right)\restr(\Sigma\cap(x+\mathcal G_{2^n}))
=\int_J(\partial\tilde\flux_n)\restr(\Sigma\cap(x+\mathcal G_{2^n}))\,\d\mu(\tilde\flux)
=\int_J\partial\tilde\flux_n\,\d\mu(\tilde\flux)
=\partial\flux_n
\end{equation*}
so that $\partial\flux_n$ is $0$-rectifiable, and
\begin{equation*}
\|\partial\flux\|_\meas
=\int_J\|\partial\tilde\flux\|_\meas\,\d\mu(\tilde\flux)
\geq\int_J\|\partial\tilde\flux_n\|_\meas\,\d\mu(\tilde\flux)
\geq\|\partial\flux_n\|_\meas\,.
\qedhere
\end{equation*}
\end{proof}

\begin{proof}[Proof of \cref{thm:polyhedralBoundary}]
Again use Smirnov's decomposition theorem to decompose $\flux$ into simple oriented curves,
\begin{equation*}
\flux=\int_J\tilde\flux\,\d\mu(\tilde\flux)
\quad\text{ with }\quad
\|\partial\flux\|_\meas=\int_J\|\partial\tilde\flux\|_\meas\,\d\mu(\tilde\flux)\,.
\end{equation*}
Denote the parameterization associated with a simple oriented curve $\tilde\flux$ by $\gamma_{\tilde\flux}:[0,1]\to\Omega$.
Since $\partial\flux$ is rectifiable by assumption, there is a countable set $S$ of points with $\partial\flux=\partial\flux\restr S$.
This implies $\gamma_{\tilde\flux}(0),\gamma_{\tilde\flux}(1)\in S$ for $\mu$-almost all $\tilde\flux$.
For each $x\in S$ and $t=0,1$ we now introduce the set
\begin{equation*}
J_x^t=\left\{\tilde\flux\in J\,\middle|\,x=\gamma_{\tilde\flux}(t)\right\}\,,
\end{equation*}
which due to $t\in\{0,1\}$ does not depend on the particular choice of parameterizations $\gamma_{\tilde\flux}$ for curves $\tilde\flux$.
Note that each $J_x^t$ is $\mu$-measurable
(indeed, it is the preimage of $x$ under the mapping $\tilde\flux\mapsto\gamma_{\tilde\flux}(t)$ with $t=0,1$,
which is continuous with respect to the underlying topology on the space of simple oriented curves, the weak-$\ast$ topology).
Also note that $J_x^0\cap J_y^1$ is disjoint from $J_z^0\cap J_w^1$ whenever $(x,y)\neq(z,w)$ so that
\begin{equation*}
\sum_{(x,y)\in S\times S}\mu(J_x^0\cap J_y^1)
=\int_{\bigcup_{(x,y)\in S\times S}J_x^0\cap J_y^1}\,\d\mu(\tilde\flux)
=\int_{J}\,\d\mu(\tilde\flux)
=\frac12\int_J\|\partial\tilde\flux\|_\meas\,\d\mu(\tilde\flux)
=\frac12\|\partial\flux\|_\meas
<\infty\,.
\end{equation*}
Thus, since $S\times S$ is countable, it is straightforward to see that we can arrange all its elements $(x,y)\in S\times S$ in decreasing order with respect to $\mu(J_x^0\cap J_y^1)$.
Denote by $(x_i,y_i)$ the $i$\textsuperscript{th} element of $S\times S$ and define
\begin{equation*}
\flux_n
=\sum_{i=1}^n\int_{J_{x_i}^0\cap J_{y_i}^1}\tilde\flux\,\d\mu(\tilde\flux)
\end{equation*}
for $n\in\N$.
In the same manner as in the previous proof we obtain $\flux_n=\kappa_n\flux$ for a monotonically increasing sequence of $\flux$-measurable functions $\kappa_n:\Omega\to[0,1]$
as well as $\|\flux-\flux_n\|_\meas\to0$.
Furthermore, $\partial\flux_n=\partial\flux_n\restr\{x_1,\ldots,x_n,y_1,\ldots,y_n\}$ and
\begin{equation*}
\|\partial\flux_n\|_\meas\leq\sum_{i=1}^n\int_{J_{x_i}^0\cap J_{y_i}^1}\|\partial\tilde\flux\|_\meas\,\d\mu(\tilde\flux)\leq\|\partial\flux\|_\meas\,.
\qedhere
\end{equation*}
\end{proof}

\section{Weak-$\ast$ relaxation of the polyhedral $h$-mass}\label{sec:relaxation}


The strategy to prove \cref{thm:equivalence} is to first restrict to transportation costs $h$ with $h(m)\geq\alpha m$ for some $\alpha>0$ and all $m>0$ and to separately consider two cases:
If the right derivative $h'(0)$ of the transportation cost $h$ in $0$ is finite, then one can prove equivalence of both relaxations directly by construction.
Otherwise, the rectifiability theorem \cite[Thm.\,7.1]{Wh99} due to White or \cite[Prop.\,2.8]{CoRoMa17} implies that $\hMass\cdot$ is only finite on $1$-rectifiable flat chains.
In that case we employ our new approximation \namecref{thm:polyhedralBoundaryApprox} for $1$-rectifiable flat chains to reduce \cref{thm:equivalence} to the case of chains with polyhedral boundary.
This case in turn has already been solved by Chambolle, Ferrari, and Merlet \cite{ChFeMe18} (under the above condition on $h$).
The proof for general transportation cost $h$ can then be reduced to costs with $h(m)\geq\alpha m$ using a representation theorem for $\hMass\chain$. 

\begin{proof}[Proof of \cref{thm:equivalence} for $h(m)\geq\alpha m$]
First consider the case $h'(0)<\infty$.
Let $\tilde\chain_n\flatto\chain$ be a sequence of polyhedral $1$-chains with $\lim_{n\to\infty}\hMass{\tilde\chain_n}=\hMass\chain<\infty$ (if $\hMass\chain=\infty$ there is nothing to prove).
Due to our growth condition on $h$ 
we have $\alpha\mass{\tilde\chain_n}\leq\hMass{\tilde\chain_n}\to\hMass\chain$ so that the $\tilde\chain_n$ have equibounded mass.
If the boundaries $\partial\tilde\chain_n$ also have equibounded mass, then $\iota_1(\tilde\chain_n)\massfluxto\iota_1(\chain)$ and thus
\begin{equation*}
\brTptCost{\iota_1(\chain)}
\leq\liminf_{n\to\infty}\brTptCost{\iota_1(\tilde\chain_n)}
=\liminf_{n\to\infty}\hMass{\tilde\chain_n}
=\hMass\chain
\end{equation*}
as desired.
Otherwise, let $\mu_\pm^1,\mu_\pm^2,\ldots\in\measp(\Omega)$ be finite linear combinations of Dirac masses such that $\mu_-^n-\mu_+^n\weakstarto\iota_0(\partial\chain)$.
Since $\tilde\chain_n$ converges flatly, also $\partial\tilde\chain_n\flatto\partial\chain$ and thus
\begin{equation*}
\iota_0^{-1}(\mu_-^n-\mu_+^n)-\partial\tilde\chain_n\flatto0
\quad\text{as }n\to\infty\,.
\end{equation*}
Consequently, and as this null sequence in $\flatChains0(\Omega)$ is polyhedral, there exist polyhedral $1$-chains $D_n$ with $D_n\to0$ in mass such that
\begin{equation*}
\mass{\iota_0^{-1}(\mu_-^n-\mu_+^n)-\partial\tilde\chain_n-\partial D_n}\to0
\quad\text{as }n\to\infty\,.
\end{equation*}
Now define $\chain_n=\tilde\chain_n+D_n$, then both $\mass{\chain_n}$ and $\mass{\partial\chain_n}$ are equibounded, and $\chain_n\flatto\chain$.
Therefore, we have $\iota_1(\chain_n)\massfluxto\iota_1(\chain)$ and
\begin{multline*}
\brTptCost{\iota_1(\chain)}
\leq\liminf_{n\to\infty}\brTptCost{\iota_1(\chain_n)}
=\liminf_{n\to\infty}\hMass{\tilde\chain_n+D_n}\\
\leq\liminf_{n\to\infty}\hMass{\tilde\chain_n}+\hMass{D_n}
\leq\liminf_{n\to\infty}\hMass{\tilde\chain_n}+h'(0)\mass{D_n}
=\lim_{n\to\infty}\hMass{\tilde\chain_n}
=\hMass\chain\,,
\end{multline*}
where we have used that the transportation cost $h$ is subadditive (so that the $h$-mass is subadditive).

Now assume $h'(0)=\infty$.
By \cite[Thm.\,7.1]{Wh99}, $\hMass\cdot$ is only finite on $1$-rectifiable flat $1$-chains so that it suffices to show $\brTptCost{\iota_1(\chain)}\leq\hMass\chain$ for a $1$-rectifiable $\chain\in\finiteFlatChains1(\Omega)$.
By \cref{thm:polyhedralBoundaryApprox} there is a sequence $\chain_n$ converging in mass to $\chain$ such that $\hMass{\chain_n}\to\hMass\chain$ and $\partial\chain_n$ is polyhedral with equibounded mass.
Due to the equibounded mass and boundary mass we have $\iota_1(\chain_n)\massfluxto\iota_1(\chain)$
and thus, by definition of the relaxation,
\begin{equation*}
\brTptCost{\iota_1(\chain)}\leq\liminf_{n\to\infty}\brTptCost{\iota_1(\chain_n)}\,.
\end{equation*}
Since $\chain_n$ has polyhedral boundary, by Chambolle, Ferrari, and Merlet \cite[Thm.\,1.2]{ChFeMe18} we know that
\begin{equation*}
\hMass{\chain_n}=\lim_{j\to\infty}\hMass{\chain_n^j}
\end{equation*}
for a sequence $\chain_n^1,\chain_n^2,\ldots$ of polyhedral flat $1$-chains with equibounded mass and $\partial \chain_n^j=\partial \chain_n$ for all $j\in\N$.
Thus we have $\iota_1(\chain_n^j)\massfluxto\iota_1(\chain_n)$ so that
\begin{equation*}
\brTptCost{\iota_1(\chain)}
\leq\liminf_{n\to\infty}\brTptCost{\iota_1(\chain_n)}
\leq\liminf_{n\to\infty}\liminf_{j\to\infty}\brTptCost{\iota_1(\chain_n^j)}
=\liminf_{n\to\infty}\lim_{j\to\infty}\hMass{\chain_n^j}
=\liminf_{n\to\infty}\hMass{\chain_n}
=\hMass\chain
\end{equation*}
as desired.
\end{proof}

To also cover the case of general transportation costs $h$, let us first note that $h$ can be approximated by a sequence of superlinear transportation costs.
To this end, define the indicator function of a set $A$ as $\chi_A(m)=0$ if $m\in A$ and $\chi_A(m)=\infty$ else
and recall that the lower semi-continuous subadditive envelope of a function $\psi:[0,\infty)\to[0,\infty]$ is the lower semi-continuous subadditive function \cite[Def.\,5.16 and Prop.\,5.17]{Br02} defined as
\begin{equation*}
m\mapsto\sup\{\phi(m)\,|\,\phi:[0,\infty)\to[0,\infty)\text{ is lower semi-continuous subadditive with }\phi\leq\psi\}\,.
\end{equation*}

\begin{lemma}[Superlinear approximation of transportation costs]\label{thm:h_M}
Let $h$ be a transportation cost.
For $M>0$ define the transportation cost $h_M:[0,\infty)\to[0,\infty)$ to be the lower semi-continuous subadditive envelope of the function $m\mapsto h(m)+\chi_{[0,M]}(m)$.
Then there exists some $\alpha>0$ such that
\begin{align*}
\alpha m&\leq h_M(m)&\text{for all }m\geq0,\\
h(m)&\leq h_M(m)&\text{for all }m\geq0,\\
h_M(m)&\leq\tfrac{2h(M)}Mm&\text{for all }m\geq M,\\
h_M(m)&=h(M)&\text{for all }m\leq M.
\end{align*}
\end{lemma}
\begin{proof}
That $h_M$ is a transportation cost with $h_M\geq h$ as well as $h_M(m)=h(m)$ for $m\leq M$ follows directly from the properties of the lower semi-continuous subadditive envelope.
Furthermore, for $m\geq M$ let $k\in\N$ and $r\in[0,M)$ such that $m=kM+r$.
Then by the subadditivity of $h_M$ we obtain
\begin{equation*}
h_M(m)
\leq kh_M(M)+h_M(r)
=kh(M)+h(r)
\leq(k+1)h(M)
\leq2kh(M)
\leq\tfrac{2h(M)}M(kM+r)
=\tfrac{2h(M)}Mm\,.
\end{equation*}
Finally, by \cite[Thm.\,5 and its proof]{La62} we have $h(m)\geq\alpha m$ for all $m\in[0,M]$ with
\begin{equation*}
\alpha=\inf\left\{\tfrac{h(m)}m\,\middle|\,m\in\left(\tfrac M2,M\right]\right\}>0
\end{equation*}
so that $h_M(m)\geq\alpha m$ for all $m\geq0$.
\end{proof}

We further require the representation theorem for $\hMass\chain$ from \cite[Prop.\,2.32, last three bullet points of the proof]{BrWi18}. 
To state it, we use that by \cite[Thm.\,4.2]{Si08} any $\chain\in\finiteFlatChains1(\Omega)$ can be uniquely decomposed into
\begin{equation*}
\chain=\chain^\rec+\chain^\diff
\end{equation*}
with $\chain^\rec,\chain^\diff\in\finiteFlatChains1(\Omega)$ a rectifiable and a diffuse flat chain,
that is, there exists a triple $[\Sigma,\theta,m]$ of a $1$-rectifiable set $\Sigma\subset\Omega$ with approximate tangent $\theta:\Sigma\to\mathbb S^{d-1}$
and a measurable function $m:\Sigma\to\R$ such that
\begin{equation*}
\iota_1(\chain^\rec)=m\theta\hd^1\restr\Sigma\,,
\end{equation*}
while $|\iota_1(\chain^\diff)|(S)=0$ for any $1$-rectifiable set $S\subset\Omega$.

\begin{theorem}[Representation of $\hMass\chain$ \protect\cite{BrWi18}]\label{thm:representation}
Let $\chain\in\finiteFlatChains1(\Omega)$ have the decomposition $\chain=\chain^\rec+\chain^\diff$, where $\chain^\rec$ is associated with the triple $[\Sigma,\theta,m]$. Then
\begin{equation*}
\hMass{\chain}
=\int_\Sigma h(|m|)\,\d\hd^1+h'(0)\mass{\chain^\diff}\,,
\end{equation*}
where $h'(0)\in[0,\infty]$ denotes the right derivative of $h$ in $0$.
\end{theorem}

Now we are prepared to finish the proof of \cref{thm:equivalence}.

\begin{proof}[Proof of \cref{thm:equivalence}]
The case of a transportation cost $h$ with $h(m)\geq\alpha m$ for some $\alpha>0$ and all $m>0$ has already been treated before.
Thus it remains to show the result for transportation costs $h$ with $h(m)/m\to0$ as $m\to\infty$.
Let $\chain\in\finiteFlatChains1(\Omega)$, and let $h_M$ denote the transportation cost from \cref{thm:h_M} for arbitrary $M>0$.
By $h_M\geq h$ and the definition of the branched transport cost we have $\brTptCost{\iota_1(\chain)}\leq\brTptCost[h_M]{\iota_1(\chain)}$.
On the other hand, by \cref{thm:representation} we have
\begin{multline*}
\hMass[h_M]\chain
=h_M'(0)\mass{\chain^\diff}+\int_{\Sigma}h_M(|m|)\,\d\hd^1\\
=h'(0)\mass{\chain^\diff}+\int_{\Sigma}h(|m|)\,\d\hd^1+\int_{\{x\in\Sigma\,|\,|m(x)|>M\}}h_M(|m|)-h(|m|)\,\d\hd^1\\
=\hMass\chain+\int_{\{x\in\Sigma\,|\,|m(x)|>M\}}h_M(|m|)-h(|m|)\,\d\hd^1\\
\leq\hMass\chain+\tfrac{2h(M)}M\int_{\{x\in\Sigma\,|\,|m(x)|>M\}}|m|\,\d\hd^1\,\leq\hMass\chain+\tfrac{2h(M)}M\mass\chain,
\end{multline*}
where the first inequality follows from the fact that $2h(M)|m|/M\geq h_M(|m|)\geq h(|m|)\geq0$ on the set $\{|m|\geq M\}$.
Furthermore, $h_M$ satisfies the growth condition for which we have already proved equality between the $h$-mass and the branched transport cost.
Thus we can summarize
\begin{equation*}
\brTptCost{\iota_1(\chain)}
\leq\brTptCost[h_M]{\iota_1(\chain)}
=\hMass[h_M]\chain
\leq\hMass\chain+\tfrac{2h(M)}M\mass\chain\,,
\end{equation*}
and the result follows from letting $M\to\infty$.
\end{proof}

\section{Consequences}

Here we briefly mention a few implications of the previous results on generalized branched transport models.
We concentrate on models (which we call \emph{admissible} below) in which the generalized branched transport cost metrizes weak-$\ast$ convergence.

\begin{definition}[Admissible transportation cost]
A transportation cost $h$ is \emph{admissible} if there exists a concave function $\beta:[0,\infty)\to[0,\infty)$
with $\int_0^1\frac{\beta(m)}{m^{2-1/d}}\,\d m<\infty$.
\end{definition}

\begin{remark}[Metrization property]
By \cite[Cor.\,2.24]{BrWi18} the admissibility condition on $h$ implies that the generalized branched transport cost
\begin{equation*}
d_{\mathbb J_h}(\mu_+,\mu_-)=\min\left\{\brTptCost\flux\,\middle|\,\flux\text{ is mass flux with }\partial\flux=\mu_--\mu_+\right\}
\end{equation*}
between two measures $\mu_+,\mu_-\in\measp(\Omega)$ metrizes weak-$\ast$ convergence on the set of probability measures.
\end{remark}

Similarly to \cite{ChFeMe18} we now show that one may also prescribe the boundary during the relaxation.

\newcommand{\hMassMu}[2][h]{\mathbb M_{#1}^{(\mu_\pm^n)}(#2)}
\newcommand{\brTptMu}[2][h]{\mathbb J_{#1}^{(\mu_\pm^n)}(#2)}
\begin{theorem}[Relaxation under prescribed boundary]\label{thm:relaxationFixedBoundary}
Let $\mu_+,\mu_-\in\measp(\Omega)$ with equal mass and fix arbitrary sequences $\mu_\pm^1,\mu_\pm^2,\ldots\in\measp(\Omega)$ of finite linear combinations of Dirac masses with $\mu_+^n(\Omega)=\mu_-^n(\Omega)$ and $\mu_\pm^n\weakstarto\mu_\pm$ as $n\to\infty$.
For any flat $1$-chain $\chain\in\finiteFlatChains1(\Omega)$ with $\iota_0(\partial\chain)=\mu_--\mu_+$ and an admissible transportation cost $h$ we have
$\hMass\chain=\brTptCost{\iota_1(\chain)}=\hMassMu\chain=\brTptMu{\iota_1(\chain)}$ for
\begin{align*}
\brTptMu{\flux}
&=\inf\left\{\liminf_{n\to\infty}\brTptCost{\flux_n}\,\middle|\,\flux_n\in\massFluxes(\Omega)\text{ polyhedral, }\flux_n\weakstarto\flux\text{ as }n\to\infty,\,\partial\flux_n=\mu_-^n-\mu_+^n\right\}\,,\\
\hMassMu\chain
&=\inf\left\{\liminf_{n\to\infty}\hMass{\chain_n}\,\middle|\,\chain_n\in\flatChains1(\Omega)\text{ polyhedral, }\chain_n\flatto\chain\text{ as }n\to\infty,\,\partial\chain_n=\iota_0^{-1}(\mu_-^n-\mu_+^n)\right\}\,.
\end{align*}
\end{theorem}
\begin{proof}
By definition we have
\begin{equation*}
\hMass\chain
=\brTptCost{\iota_1(\chain)}
\leq\hMassMu\chain
\leq\brTptMu{\iota_1(\chain)}\,,
\end{equation*}
so it suffices to show $\brTptMu{\iota_1(\chain)}\leq\brTptCost{\iota_1(\chain)}$.
Abbreviate $\flux=\iota_1(\chain)$ and consider a sequence $\bar\flux_1,\bar\flux_2,\ldots$ of polyhedral fluxes with $\bar\flux_n\massfluxto\flux$ and $\brTptCost{\bar\flux_n}\to\brTptCost\flux$ as $n\to\infty$.
Next, by the Jordan Decomposition Theorem we can decompose $\partial\bar\flux_n-\mu_-^n+\mu_+^n=\nu_+^n-\nu_-^n$ with $\nu_\pm^n\in\measp(\Omega)$.
Obviously, $\nu_-^n-\nu_+^n\weakstarto0$ as $n\to\infty$ so that by \cite[Cor.\,2.24]{BrWi18} there exists a sequence $\hat\flux_1,\hat\flux_2,\ldots$ of equibounded mass fluxes
with $\partial\hat\flux_n=\nu_-^n-\nu_+^n$ and $\brTptCost{\hat\flux_n}\to0$.
Since $\nu_+^n$ and $\nu_-^n$ are finite linear combinations of Dirac masses, the $\hat\flux_n$ can be chosen as polyhedral fluxes.
Finally define the sequence $\flux_n=\bar\flux_n+\hat\flux_n$, $n\in\N$, of polyhedral fluxes,
then $\partial\flux_n=\mu_-^n-\mu_+^n$ as well as $\flux_n\massfluxto\flux$ and $\brTptCost{\flux_n}\leq\brTptCost{\bar\flux_n}+\brTptCost{\hat\flux_n}\to\brTptCost\flux$ as $n\to\infty$, as desired.
\end{proof}

As also emphasized in \cite{ChFeMe18}, the latter result is particularly useful for the development of phasefield approximations of generalized branched transport or minimal $h$-mass problems.
Indeed, when proving $\Gamma$-convergence of a phasefield functional to the minimal $h$-mass problem with prescribed boundary,
a recovery sequence can typically only be constructed for polyhedral fluxes, in particular with polyhedral boundary.
The above result implies that this is indeed sufficient.

Finally we state that the generalized branched transport problem and the problem of minimizing the $h$-mass are equivalent.

\begin{theorem}[Branched transport problem and minimal $h$-mass]
Let $h$ be an admissible transportation cost and $\mu_+,\mu_-\in\measp(\Omega)$ with equal mass, then
\begin{equation*}
\min\left\{\brTptCost\flux\,\middle|\,\flux\in\massFluxes(\Omega),\,\partial\flux=\mu_--\mu_+\right\}
=\min\left\{\hMass\chain\,\middle|\,\chain\in\flatChains1(\Omega),\,\partial\chain=\iota_0^{-1}(\mu_--\mu_+)\right\}\,,
\end{equation*}
and the minimizers of both problems are related by $\iota_1$.
\end{theorem}
\begin{proof}
Let us abbreviate
\begin{align*}
d_{\mathbb J_h}(\mu_+,\mu_-)
&=\inf\left\{\brTptCost\flux\,\middle|\,\flux\in\massFluxes(\Omega),\,\partial\flux=\mu_--\mu_+\right\}\,,\\
d_{\mathbb M_h}(\mu_+,\mu_-)
&=\inf\left\{\hMass\chain\,\middle|\,\chain\in\flatChains1(\Omega),\,\partial\chain=\iota_0^{-1}(\mu_--\mu_+)\right\}\,.
\end{align*}
The existence of minimizers for $d_{\mathbb J_h}(\mu_+,\mu_-)$ is shown in \cite[Cor.\,2.20]{BrWi18},
and since each minimizer $\flux$ for $d_{\mathbb J_h}(\mu_+,\mu_-)$ induces a competitor $\iota_1^{-1}(\flux)$ for $d_{\mathbb M_h}(\mu_+,\mu_-)$ with $\hMass{\iota_1^{-1}(\flux)}=\brTptCost\flux$,
we have $d_{\mathbb J_h}(\mu_+,\mu_-)\geq d_{\mathbb M_h}(\mu_+,\mu_-)$ and only need to show the opposite inequality.
To this end consider two sequences $\mu_\pm^1,\mu_\pm^2,\ldots$ of nonnegative finite linear combinations of Dirac masses such that $\mu_+^n(\Omega)=\mu_-^n(\Omega)=\mu_+(\Omega)$ and $\mu_\pm^n\weakstarto\mu_\pm$ as $n\to\infty$,
and abbreviate $M=\mu_+(\Omega)$.

Let us first restrict ourselves to the case where $h'(0)<\infty$.
By the triangle inequality (which follows from the subadditivity of $\hMass\cdot$) we have
\begin{equation*}
d_{\mathbb M_h}(\mu_+,\mu_-)
\geq d_{\mathbb M_h}(\mu_+^n,\mu_-^n)-d_{\mathbb M_h}(\mu_+^n,\mu_+)-d_{\mathbb M_h}(\mu_-,\mu_-^n)\,,
\end{equation*}
where without loss of generality we may assume $\flatnorm{\iota_0^{-1}(\mu_+^n-\mu_+)}+\flatnorm{\iota_0^{-1}(\mu_--\mu_-^n)}\leq\frac1n$
as well as $d_{\mathbb M_h}(\mu_+^n,\mu_+)+d_{\mathbb M_h}(\mu_-,\mu_-^n)<\frac1n$ due to the admissibility of $h$.
Now let $\chain_n\in\flatChains1(\Omega)$ with $\partial\chain_n=\iota_0^{-1}(\mu_-^n-\mu_+^n)$ and $\hMass{\chain_n}\leq d_{\mathbb M_h}(\mu_+^n,\mu_-^n)+\frac1n$.
By definition of $\hMass\cdot$ there exists a polyhedral $1$-chain $\bar\chain_n$ with $\hMass{\bar\chain_n}\leq\hMass{\chain_n}+\frac1n$
and $\flatnorm{\bar\chain_n-\chain_n}\leq\frac1n$ as well as $\flatnorm{\partial\bar\chain_n-\partial\chain_n}\leq\frac1n$.
The latter implies $\flatnorm{\partial\bar\chain_n-\iota_0^{-1}(\mu_--\mu_+)}\leq\frac2n$
and thus the existence of a flat $1$-chain $S_n\in\finiteFlatChains1(\Omega)$ with $\mass{\partial(\bar\chain_n+S_n)-\iota_0^{-1}(\mu_--\mu_+)}+\mass{S_n}\leq\frac2n$.
Letting $\bar\mu_+^n,\bar\mu_-^n\in\measp(\Omega)$ be the positive and the negative part of $\iota_0(\partial(\bar\chain_n+S_n))$, we can summarize
\begin{multline*}
d_{\mathbb M_h}(\mu_+,\mu_-)
\geq d_{\mathbb M_h}(\mu_+^n,\mu_-^n)-\tfrac1n
\geq\hMass{\chain_n}-\tfrac2n
\geq\hMass{\bar\chain_n}-\tfrac3n
\geq\hMass{\bar\chain_n}+h'(0)\mass{S_n}-\tfrac{3+2h'(0)}n\\
\geq\hMass{\bar\chain_n+S_n}-\tfrac{3+2h'(0)}n
=\brTptCost{\iota_1^{-1}(\bar\chain_n+S_n)}-\tfrac{3+2h'(0)}n
\geq d_{\mathbb J_h}(\bar\mu_+^n,\bar\mu_-^n)-\tfrac{3+2h'(0)}n\,.
\end{multline*}
For $n\to\infty$ we obtain the desired inequality
if we can show $\lim_{n\to\infty}d_{\mathbb J_h}(\bar\mu_+^n,\bar\mu_-^n)=d_{\mathbb J_h}(\mu_+,\mu_-)$.
Note that $\bar\mu_\pm^n\to\mu_\pm$ strongly and let $\hat\mu_+^n,\hat\mu_-^n\in\measp(\Omega)$ be the positive and the negative part of $\mu_+-\bar\mu_+^n-\mu_-+\bar\mu_-^n$.
Then $d_{\mathbb J_h}(\hat\mu_+^n,\hat\mu_-^n)\leq h'(0)W_1(\hat\mu_+^n,\hat\mu_-^n)\leq h'(0)\mathrm{diam}(\Omega)\|\hat\mu_+^n\|_\meas\to0$ as $n\to\infty$,
where $W_1$ denotes the Wasserstein-$1$ distance and $\mathrm{diam}(\Omega)$ denotes the intrinsic diameter of $\Omega$ (the largest geodesic distance between two points in $\Omega$).
Thus, taking the limit $n\to\infty$ in
\begin{equation*}
d_{\mathbb J_h}(\mu_+,\mu_-)
\leq d_{\mathbb J_h}(\bar\mu_+^n,\bar\mu_-^n)+d_{\mathbb J_h}(\hat\mu_+^n,\hat\mu_-^n)
\end{equation*}
yields the desired result.

Now consider the case $h'(0)=\infty$.
For any $N\in\N$ define $h^N$ to be the lower semi-continuous subadditive envelope of $m\mapsto\min\{h(m),Nm\}$.
Using $(h^N)'(0)<\infty$ we thus obtain
\begin{equation*}
d_{\mathbb M_h}(\mu_+,\mu_-)
\geq d_{\mathbb M_{h^N}}(\mu_+,\mu_-)
=d_{\mathbb J_{h^N}}(\mu_+,\mu_-)\,.
\end{equation*}
Let further $\flux_N\in\massFluxes(\Omega)$ denote a minimizer for the right-hand side so that via \cref{thm:relaxationFixedBoundary} we have
\begin{equation*}
d_{\mathbb M_h}(\mu_+,\mu_-)
\geq d_{\mathbb J_{h^N}}(\mu_+,\mu_-)
=\brTptCost[h^N]{\flux_N}
=\brTptMu[h^N]{\flux_N}
=\lim_{n\to\infty}\brTptCost[h^N]{\flux_N^n}\,,
\end{equation*}
where $\flux_N^n$ is a polyhedral flux with $\partial\flux_N^n=\mu_-^n-\mu_+^n$.
By \cite[Lem.\,2.5]{BrWi18} we can reduce the right-hand side even further by replacing $\flux_N^n$ with an acyclic flux $\tilde\flux_N^n$ of the same boundary,
which by \cite[Lem.\,2.9]{BrWi18} has multiplicity bounded by $M$.
Since by \cite[Thm.\,5 and its proof]{La62} we have $h^N(m)\geq\alpha_Nm\geq\alpha_1m$ for all $m\in[0,M]$ with
\begin{equation*}
\alpha_N=\inf\left\{\tfrac{h^N(m)}m\,\middle|\,m\in\left(\tfrac M2,M\right]\right\}>0\,,
\end{equation*}
we see $d_{\mathbb M_h}(\mu_+,\mu_-)\geq\lim_{n\to\infty}\brTptCost[h^N]{\tilde\flux_N^n}\geq\lim_{n\to\infty}\alpha_1\|\tilde\flux_N^n\|_\meas$.
Hence, the $\tilde\flux_N^n$ have equibounded mass and converge weakly as mass fluxes (up to a subsequence) to some $\tilde\flux_N$ with $\|\tilde\flux_N\|_\meas\leq d_{\mathbb M_h}(\mu_+,\mu_-)/\alpha_1$.
Summarizing, we obtain
\begin{equation*}
d_{\mathbb M_h}(\mu_+,\mu_-)
\geq\lim_{n\to\infty}\brTptCost[h^N]{\tilde\flux_N^n}
\geq\brTptCost[h^N]{\tilde\flux_N}
\end{equation*}
for all $N\in\N$, where $\tilde\flux_N$ with $\partial\tilde\flux_N=\mu_+-\mu_-$ has mass bounded by $d_{\mathbb M_h}(\mu_+,\mu_-)/\alpha_1$.
Again restricting to a subsequence (still indexed by $N$) we have $\tilde\flux_N\massfluxto\flux$ for some $\flux\in\massFluxes(\Omega)$ and thus
\begin{equation*}
d_{\mathbb M_h}(\mu_+,\mu_-)
\geq\lim_{N\to\infty}\brTptCost[h^N]{\tilde\flux_N}
\geq\liminf_{N\to\infty}\brTptCost[h^L]{\tilde\flux_N}
\geq\brTptCost[h^L]{\flux}
=\hMass[h^L]{\iota_1^{-1}(\flux)}
\end{equation*}
for any $L\in\N$.
Using the representation \cref{thm:representation} and the Monotone Convergence Theorem as $L\to\infty$ we arrive at
\begin{equation*}
d_{\mathbb M_h}(\mu_+,\mu_-)
\geq\hMass{\iota_1^{-1}(\flux)}
=\brTptCost\flux
\geq d_{\mathbb J_h}(\mu_+,\mu_-)
\end{equation*}
as desired.

\end{proof}

\section*{Acknowledgements}
B.W.'s research was supported by the Alfried Krupp Prize for Young University Teachers awarded by the Alfried Krupp von Bohlen und Halbach-Stiftung.
A.M.\ acknowledges partial support from the section GNAMPA of INdAM.

\bibliographystyle{plain}
\bibliography{notes}

\end{document}